\documentclass[12pt]{amsart}

\usepackage{latexsym,amsfonts,amssymb,epsfig,verbatim}
\usepackage{amsmath, latexsym, graphics, textcomp,psfrag}
\usepackage{float}
\usepackage{amsthm}
\usepackage{pinlabel,enumitem,mathrsfs}

\newcommand{\nc}{\newcommand}

\nc{\dmo}{\DeclareMathOperator}
\nc{\nt}{\newtheorem}

\nt{theorem}{Theorem}
\nt{lemma}[theorem]{Lemma}
\nt{prop}[theorem]{Proposition}
\nt{fact}[theorem]{Fact}
\nt{question}[theorem]{Question}
\nt{cor}[theorem]{Corollary}
\nt{problem}[theorem]{Problem}
\nt{conjecture}[theorem]{Conjecture}

\dmo{\colim}{colim}

\nc{\Z}{\mathbb{Z}}
\nc{\R}{\mathbb{R}}
\nc{\N}{\mathbb{N}}
\nc{\Q}{\mathbb{Q}}

\nc{\A}{\mathcal{A}}
\nc{\C}{\mathcal{C}}
\nc{\Cbar}{\overline{\mathcal{C}}}
\nc{\B}{\mathcal{B}}
\nc{\Bbar}{\overline{\mathcal{B}}}
\nc{\M}{\mathcal{M}}
\nc{\F}{\mathcal{F}}
\nc{\PP}{\mathcal{P}}

\nc {\cln}{\,\colon\!} 
\nc {\bdy}{\partial}

\nc{\s}{\mathscr{C}}

\dmo{\Isom}{Isom}
\dmo{\Homeo}{Homeo}
\dmo{\Diff}{Diff}
\dmo{\Mod}{M}
\dmo{\Aut}{Aut}
\dmo{\Lk}{Lk}
\dmo{\PMod}{PMod}
\dmo{\Sp}{Sp}
\dmo{\Push}{Push}
\dmo{\Forget}{Forget}
\dmo{\supp}{supp}
\dmo{\Surger}{Surger}
\dmo{\Drain}{Drain}
\dmo{\Teich}{Teich}
\dmo{\Tor}{T}
\nc{\I}{\Tor}

\nc{\p}[1]{\smallskip\noindent{{\bf #1}}}

\include{diagram}

\begin{document}

\input{epsf.sty}


\title{Generating the Torelli group}

\author{Allen Hatcher}

\author{Dan Margalit}

\keywords{Torelli group, bounding pair maps}

\subjclass[2000]{Primary: 20F36; Secondary: 57M07}

\thanks{The second author gratefully acknowledges support from the National Science Foundation and the Sloan Foundation.}

\begin{abstract}
We give a new proof of the theorem of Birman--Powell that the Torelli subgroup of the mapping class group of a closed orientable surface of genus at least $3$ is generated by simple homeomorphisms known as bounding pair maps.  The key ingredient is a proof that the subcomplex of the curve complex of the surface spanned by curves within a fixed homology class is connected. 
\end{abstract}

\maketitle

\section{Introduction}

The \emph{mapping class group} of a closed connected orientable surface $S$ is $ \Mod(S) = \pi_0(\Diff^+(S)) $, the group of isotopy classes of orientation-preserving diffeomorphisms of $S$. Perhaps the simplest type of isotopically nontrivial diffeomorphism of $S$ is a \emph{Dehn twist} along an embedded closed curve. This is a diffeomorphism supported on an annular neighborhood of the curve, the effect of the diffeomorphism on arcs crossing the annulus being to twist these arcs around the annulus as shown in the following figure:
\begin{center}
\includegraphics[scale=1]{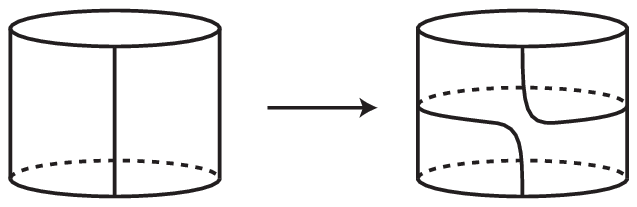}
\end{center}
In the 1920s Dehn proved that $\Mod(S)$ is generated by these twists, although this result was only published a decade later in \cite[\S 10]{dehn}; an English translation can be found in \cite[Paper 8]{stillwell}. 

The group $\Mod(S)$ has a natural action on $H_1(S)=H_1(S;\Z)$, and the kernel of this action is known as the \emph{Torelli group}, for which we use the notation $\I(S)$. 
In this paper we will be interested in the analogue of Dehn's theorem for the Torelli group.  

Certain Dehn twists belong to the Torelli group, namely the twists along curves that separate $S$.  This can be seen by observing that for each separating curve $c$, a basis for $H_1(S)$ can be chosen consisting of curves disjoint from $c$.  Conversely, it is easy to see that if $c$ is nonseparating, a twist along $c$ acts nontrivially on the homology class of a curve that crosses $c$ exactly once.   

The next-simplest element of $\I(S)$ after a separating twist is a \emph{bounding pair map}, which is the composition of a twist along a nonseparating curve $c$ and an inverse twist along another nonseparating curve $d$ disjoint from $c$ but representing the same homology class as $c$, so ${c\cup d}$  separates $S$ into two subsurfaces having ${c\cup d}$ as their common boundary.  

\vspace{-16pt}
\begin{center}
\includegraphics[scale=1]{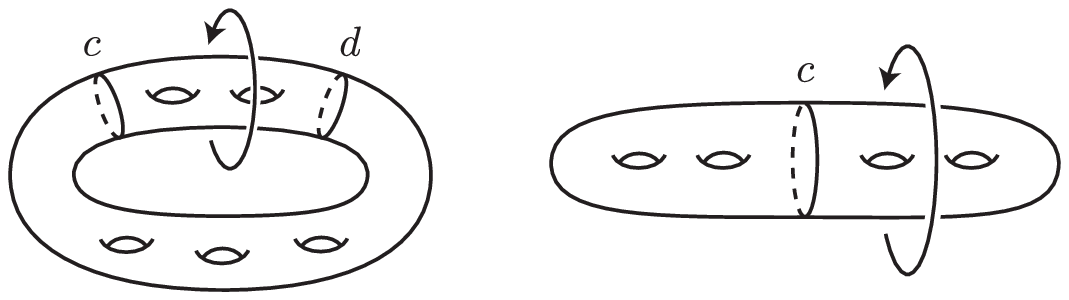}
\end{center}

\noindent
A bounding pair map, like a separating twist, can be realized by a motion of an embedding of $S$ in $\R^3$ in which a subsurface is rotated through $360$ degrees, a subsurface bounded by $c$ in the case of a separating twist and by $c$ and $d$ in the case of a bounding pair map. 
One can see that a bounding pair map acts trivially on homology by noting that a basis for $H_1(S)$ can be chosen to consist of curves disjoint from $c$ and $d$ except for one curve that crosses each of $c$ and $d$ exactly once, and it is easy to see that this curve is taken to a homologous curve.  

\begin{theorem}[Birman--Powell]
\label{main}
The Torelli group  $\I(S)$ is generated by separating twists and bounding pair maps.
\end{theorem}

Nontrivial separating twists exist only when the genus $g$ of $S$ is at least $2$, and nontrivial bounding pair maps exist only when $g\ge 3$.  Thus when $g=1$ the Torelli group is trivial, a fact known long before the Birman--Powell theorem, and when $g=2$ the theorem says that $\I(S)$ is generated by separating twists. When $g\ge 3$, separating twists do not generate all of $\I(S)$ but only a subgroup of infinite index known as the Johnson kernel \cite{djabelian}.  On the other hand, it is easy to express separating twists as products of bounding pair maps when $g\ge3$, as we recall in Proposition~\ref{seps}, so bounding pair maps alone generate $\I(S)$ when $g\ge 3$.  It is also easy to express all bounding pair maps in terms of those where the two curves that specify the map cut off a genus $1$ subsurface of $S$; see Proposition~\ref{bps}.

\smallskip

In this paper we give a proof of the Birman--Powell theorem that is in line with the standard proofs of Dehn's theorem on $\Mod(S)$ (see, e.g., \cite[Theorem 4.1]{fm}).  Dehn in fact found a finite set of twist maps that generate $\Mod(S)$, and the standard proofs of Dehn's theorem usually prove this as well.  Finding a finite set of bounding pair maps that generate $\I(S)$ when $g\ge3$ is more difficult, and was first done by Johnson \cite{dj1}.  We do not attempt to reprove the finite generation here.  Recently Putman \cite{cubic} has greatly improved Johnson's finite generation result by showing that the number of generators can be reduced from Johnson's exponential function of $g$ to a cubic function of $g$.

The genus $2$ case of the Birman--Powell theorem was not stated explicitly in their original papers, but can easily be deduced from their methods.  The genus $2$ case is exceptional not only in the types of generators that are needed, but also in the fact that $\I(S)$ is not finitely generated in this case, a result due originally to McCullough and Miller \cite{mcm} and subsequently improved by Mess \cite{gm} who showed that $\I(S)$ is a free group on a countably infinite set of twist generators. Another more recent proof of this can be found in \cite{bbm} and we give a version of this proof at the end of Section~\ref{sec:2}.

Before sketching the idea for the new proof of the Birman--Powell theorem, let us say a few words about the two prior proofs in the literature.

\medskip
\noindent
\emph{First proof: Birman and Powell 1970s}.  This starts with a fact from group theory.  Suppose we have a short exact sequence of groups 
\[ 1 \longrightarrow A \longrightarrow B \stackrel{\pi}{\longrightarrow} C \longrightarrow 1. \]
Let $\{b_1,\dots,b_n\}$ be a finite set of generators for $B$, and suppose we have a presentation for $C$ in terms of the generating set $\{ \pi(b_i) \}$.  The relators for the presentation of $C$ are words in the $\pi(b_i)$, and to each such relator $w$, the corresponding product of the $b_i$ gives an element $\widetilde w$ of $A$ if we identify $A$ with the kernel of $\pi$.  It is then a basic fact that the collection of all such $\widetilde w$, together with their conjugates in $B$, forms a generating set for $A$ (see the proof of \cite[Theorem 2.1]{mks}).

In 1961 Klingen \cite{klingen} gave an algorithm for finding a presentation of $\Sp(2g,\Z)$ and ten years later Birman \cite[Theorem 1]{jb} used this algorithm to give an explicit finite presentation for $\Sp(2g,\Z)$.  Birman's presentation therefore gave a generating set for $\I(S)$ as above, but it was not immediately clear how to interpret the generators geometrically.  In 1978, Powell recognized Birman's generators as Dehn twists about separating curves and bounding pair maps, or products of these \cite[Theorem 2]{jp}.  Neither the Birman paper nor the Powell paper contain complete details, as in both cases the required calculations are lengthy and technical.

\medskip
\noindent
\emph{Second proof: Putman 2007}.  Here the starting point is the fact that if a group $G$ acts on a simply-connected simplicial complex $X$ in such a way that the quotient of $X$ by $G$ is simply connected, then $G$ is generated by stabilizers of simplices.

Putman \cite{ap} applied this strategy to the action of $\I(S)$ on the curve complex $\C(S)$.  This is the simplicial complex whose vertices are the isotopy classes of nontrivial simple closed curves in $S$, with simplices corresponding to collections of disjoint curves. Harer had shown that $\C(S)$ is $(2g-3)$-connected, hence simply connected when $g\ge2$.  Putman then showed that $\C(S)/\I(S)$ is simply connected when $g\ge2$. 

By the fact about group actions stated above, this reduces the problem to understanding stabilizers of simplices of $\C(S)$.  Just as in the proof that $\Mod(S)$ is generated by Dehn twists, this step can be accomplished by induction on $g$ using the Birman exact sequence. 

\medskip
\noindent
\emph{A new proof.} This also proceeds by induction on genus, but is based on a different fact about group actions on complexes that only requires the complexes to be connected.  The complex we use is the subcomplex $\C_x(S)$ of $\C(S)$ spanned by curves that can be oriented so as to represent some fixed primitive class $x \in H_1(S)$.  The key fact is therefore:

\begin{theorem}[Putman]
\label{connected}
For $g \geq 3$, the complex $\C_x(S)$ is connected.
\end{theorem}

Putman's proof of this in \cite[Theorem 1.9]{apnote} uses Johnson's explicit finite generating set for $\I(S)$, which in turn depends on the Birman--Powell theorem.  The idea of the new approach is to reverse these dependencies to give a proof of Theorem~\ref{connected} from scratch and deduce the Birman--Powell theorem from this.

The induction for the new proof of the Birman--Powell theorem starts with the case $g=2$.  The complex $\C_x(S)$ has dimension $0$ in this case and is not connected, so instead we use a larger complex $\B_x(S)$ that appears in the proof of Theorem~\ref{connected}.  It was shown in \cite{bbm} that $\B_x(S)$ is contractible for all $g$, and we reprove this here.  For $g\ge 3$ we only need that $\B_x(S)$ is connected, but when $g=2$ the complex $\B_x(S)$ is $1$-dimensional and we need that it is a tree so that we can use basic facts about groups acting on trees.

\medskip

Our goal is to give a self-contained proof of the Birman--Powell theorem, and so this paper contains a number of proofs of known results, even when our proofs are not essentially different from the existing proofs.  

\medskip

\medskip
\noindent
\emph{Outline of the paper.}  In Section~\ref{bx} we recall from \cite{bbm} the construction of a complex $\B_x(S) \supset \C_x(S)$ whose points are isotopy classes of oriented, weighted multicurves (collections of finitely many disjoint curves) in $S$ representing the homology class $x$, and we show that $\B_x(S)$ is contractible.  In Section~\ref{sec:connect} we  prove Theorem~\ref{connected} by showing the map $\pi_0\C_x(S)\to\pi_0\B_x(S)$ is injective.  This is where most of the novelty of the paper occurs.  The inductive step in the proof of the Birman--Powell theorem is given in Section~\ref{sec:proof}.  Finally, in Section~\ref{sec:2} we complete the proof, using $\B_x(S)$ directly to handle the base case of the induction, genus~$2$.

\medskip
\noindent
\emph{Acknowledgments.} We would like to thank Tara Brendle, Leah Childers, Thomas Church, John Etnyre, Chris Leininger, Andy Putman, Saul Schleimer, and the referee for helpful comments and discussions.


\section{Representing homology classes by multicurves}
\label{bx}

In this section we reformulate some constructions from \cite{bbm} involving the representations of elements of $H_1(S)$ by linear combinations of disjoint oriented curves in $S$.  

By a \emph{multicurve} in $S$ we mean a collection, possibly empty, of finitely many disjoint simple closed curves in $S$, none of which bounds a disk in $S$ and no two of which bound an annulus.  Usually we will not distinguish between a multicurve and its isotopy class.  If orientations are specified for each curve $c_i$ in a multicurve $c=c_1 \cup \cdots \cup c_n$, then a linear combination $\sum_ik_ic_i$ with coefficients $k_i\in \Z$ determines a class $[\sum_ik_ic_i]$ in $H_1(S)$.  If we allow coefficients $k_i\in\R$, then $\sum_ik_ic_i$ gives a class in $H_1(S;\R)$.  By reorienting the curves $c_i$ if necessary we can assume $k_i\ge 0$ for each $i$. 
The linear combinations $\sum_ik_ic_i$ then correspond to points in the first orthant $[0,\infty)^n$ in $\R^n$. For each oriented multicurve $c$ we have a corresponding orthant $O(c)$, and we can form a space $\A(S)$ by starting with the disjoint union of all such orthants $O(c)$, one for each isotopy class of oriented multicurves $c$, and then identifying the faces obtained by setting some coefficients $k_i$ equal to $0$ with the orthants corresponding to the multicurves obtained by deleting the corresponding curves $c_i$.  (When $c$ is empty, the orthant $O(c)$ reduces to just the origin in $\R^0$, so the origins of all the orthants $O(c)$ are identified.)

The natural map $h\cln\A(S)\to H_1(S;\R)$ sending a weighted oriented multicurve $\sum_i k_i c_i$ to its homology class is linear on each orthant $O(c)$.  For a nonzero class $x \in H_1(S;\R)$ we define $\A_x(S) =h^{-1}(x)$.  This is a cell complex whose cells are the intersections of orthants $O(c)=[0,\infty)^n$ with affine planes in $\R^n$.  These `cells' $E(c)=O(c)\cap h^{-1}(x)$ can be noncompact, so $A_x(S)$ will not be a cell complex in the usual sense, but something more general.  To guarantee that the cells $E(c)$ are compact we need to impose a further condition on the oriented multicurves $c$, namely that if we translate the affine plane that determines $E(c)$ until this plane passes through the origin, which amounts to taking $ O(c) \cap h^{-1}(0)$, then $ O(c) \cap h^{-1}(0)= \{0\}$.  In other words, no nontrivial linear combination $\sum_i k_i c_i$ with each $k_i\ge 0$ represents $0$ in $H_1(S;\R)$.  Taking only orthants $O(c)$ for $c$ satisfying this extra condition yields a subspace $\B(S)$ of $\A(S)$.  The corresponding subspace $\B_x(S)$ of $\A_x(S)$ has the structure of a cell complex in the usual sense, with cells $E(c)$ that are compact convex polyhedra. 

There is another way to characterize the compact cells $E(c)$ that is somewhat more geometric:

\begin{prop}
\label{compact}
A cell $E(c)$ is compact if and only if no submulticurve of $c$ with the induced orientation from $c$ is a boundary, representing $0$ in $H_1(S)$.
\end{prop}

We call an oriented multicurve satisfying this property \emph{reduced}.  As we will see in the proof, an equivalent condition is that no submulticurve with its induced orientation is the oriented boundary of an oriented subsurface of $S$.

\begin{proof}
As noted earlier, $E(c)$ is compact if and only if no nontrivial linear combination $\sum_i k_i c_i$ with each $k_i\ge 0$ is trivial in $H_1(S)$.  This obviously implies that no submulticurve is a boundary, either in the homological sense or the geometric sense of bounding an oriented subsurface of $S$.  

For the converse, suppose that some nontrivial sum $\sum_i k_i c_i$, with each $k_i\in [0,\infty)$, is a boundary, say $\sum_i k_i c_i = \bdy\bigl(\sum_j l_j R_j\bigr)$ where the $R_j$'s are the closures of the components of $S-c$, oriented via a fixed orientation of $S$.  Since $\bdy\bigl(\sum_j R_j\bigr)=0$, we can add a large constant $l$ to each $l_j$ to guarantee that $l_j>0$ for all $j$.  Let $R$ be the union of the $R_j$'s with maximal $l_j$.  Since $\sum_i k_i c_i$ is nontrivial, the surface $R$ is a proper subsurface of $S$.  Then the equation $\sum_i k_i c_i = \bdy\bigl(\sum_j l_j R_j\bigr)$ implies that $\bdy R$ is a nonzero linear combination of the oriented curves $c_i$ with each coefficient equal to $0$ or $1$.  Thus $c$ has a null-homologous submulticurve, in fact a submulticurve that bounds a subsurface of $S$. 
\end{proof}

A multicurve $c$ has a dual graph $G(c)$ whose vertices correspond to components of $S - c$ and whose edges correspond to components of $c$.  If $c$ is an oriented multicurve, then $G(c)$ becomes an oriented graph by fixing an orientation of $S$ and a rule for passing from an orientation of $c$ to a transverse orientation.  The condition for $c$ to be reduced can be translated into a condition on $G(c)$:  

\begin{prop}
An oriented multicurve $c$ is reduced if and only if its dual graph $G(c)$ is recurrent: Every edge of $G(c)$ lies in a loop consisting of a finite sequence of edges traversed in the directions given by their orientations in $G(c)$.  Such loops can be assumed to be embedded.
\end{prop}

This recurrence condition can be restated in terms of $c$ as saying that through every point of $c$ there passes a closed oriented loop in $S$ transverse to $c$ whose algebraic and geometric intersection numbers with $c$ are equal, so the loop always crosses $c$ in the same direction.

\begin{proof}
If $c$ contains a bounding submulticurve $c'$, then any closed loop in $S$ intersecting $c'$ transversely must cross $c'$ in both directions, so $G(c)$ is not recurrent.  Conversely, suppose $G(c)$ is not recurrent.  Let $\overline{G}$ be the quotient graph of $G(c)$ obtained by collapsing to a point each component of the subgraph consisting of edges that lie in closed loops of oriented edges.  Then $\overline{G}$ contains no such edges. Hence $\overline{G}$ must contain at least one ``sink" vertex whose abutting edges are all oriented toward the vertex.  These edges correspond to a bounding submulticurve of $c$.
\end{proof}

\begin{prop}
\label{dim}
The cell $E(c)$ of $\B_x(S)$ corresponding to an oriented multicurve $c$ has dimension equal to one less than the number of connected components of $S - c$.
\end{prop}

\begin{proof}
Consider cellular homology with coefficients in $\R$ for a cell structure on $S$ containing $c$ as a subcomplex.  With notation as in the proof of Proposition~\ref{compact}, the regions $R_j$ generate a subgroup $\langle R_j \rangle$ of the $2$-chains and the curves $c_i$ generate a subgroup $\langle c_i\rangle$ of the $1$-chains. The boundary map $\bdy\cln \langle R_j \rangle\to \langle c_i \rangle $ has $1$-dimensional kernel $H_2(S;\R)$ so its image has dimension one less than the number of $R_j$'s.  Cosets of this image are the planes $h^{-1}(x)$, and intersecting one of these with the orthant $O(c)$ gives the cell $E(c)$, with the stated dimension.
\end{proof}

Up until this point the class $x$ was any nonzero element of $H_1(S;\R)$, but from now on we restrict attention to classes in $H_1(S) = H_1(S;\Z)$.

\begin{prop}
For a nonzero class $x\in H_1(S)$, the coefficients of a vertex $\sum_i k_i c_i$ of $\B_x(S)$ are integers.
\end{prop}

\begin{proof}
By Proposition~\ref{dim} the multicurve $c=\cup_i c_i$ has connected complement, so for each $c_j$ there is a transverse curve $d_j$ intersecting $c_j$ once and disjoint from the other $c_i$'s. The algebraic intersection number of $d_j$ with $\sum_i k_i c_i$ is then $\pm k_j$, so if $\sum_i k_i c_i$ represents an integral homology class, $k_j$ must be an integer.
\end{proof}

\begin{prop}
\label{b contract}
Let $g \geq 1$, and let $x \in H_1(S)$ be any nonzero element.  Then the complexes $\A_x(S)$ and $\B_x(S)$ are contractible.
\end{prop}

\begin{proof}
We follow the plan of the second proof in \cite{bbm}.  First consider $\A_x(S)$.  To prove that $\A_x(S)$ is contractible we will construct a canonical ``linear" path in $\A_x(S)$ joining any two given points $\alpha$ and  $\beta$, assuming that the multicurves underlying $\alpha$ and $\beta$ have first been isotoped to intersect transversely with the minimum number of points of intersection.  With this minimality condition the configuration formed by the union of the two multicurves is unique up to isotopy of $S$, which will ensure that the construction is well defined on isotopy classes.

To a weighted multicurve $\alpha =\sum_ik_i c_i$ representing a point in $\A_x(S)$ we associate a map $f_\alpha\cln S \to S^1$ in the following way.  First, choose disjoint product neighborhoods $c_i \times [0,k_i]$ of the curves $c_i$ in $S$.  (We can assume each $k_i > 0$ by deleting any  $c_i$ with $k_i=0$.)  From these product neighborhoods we obtain a quotient map $q$ from $S$ to the graph $G(c)$ by projecting each $c_i\times[0,k_i]$ to $[0,k_i]$ and then to the corresponding edge of $G(c)$, with the complementary components of the thickened $c$ in $S$ mapping to the corresponding vertices of $G(c)$.  The weights $k_i$ determine lengths for the edges of $G(c)$ making it into a metric graph, with edges oriented via the orientation of $c$.  There is then a natural map $\varphi\cln G(c)\to S^1=\R/\Z$ defined up to rotations of $S^1$ by the condition that it is an orientation-preserving local isometry on each edge of $G(c)$.  Namely, choose a vertex of $G(c)$ and send it to an arbitrary point in $\R$.  This determines a map on adjacent edges sending them isometrically to $\R$ preserving orientations, then continue inductively for edges adjacent to the previous edges.  Loops in $G(c)$ have signed length equal to the algebraic intersection number of  lifted loops in $S$ with $x$, and these intersection numbers are integers since $x$ is an integral homology class, so when we pass to the quotient $\R/\Z$ we have a well-defined map $\varphi\cln G(c)\to S^1$.  Changing the initial vertex or its image in $\R$ has the effect of composing $\varphi$ with a rotation of $S^1$. The composition $\varphi q$ is then a map $f_\alpha\cln S\to S^1$.  This corresponds to the class in $H^1(S;\Z)$ Poincar\'e dual to $x$. We can arrange that $f_\alpha$ is a smooth map by parametrizing the annuli $c_i \times [0,k_i]$ suitably.

For a second point $\beta$ in $\A_x(S)$ we choose annular neighborhoods of its curves that intersect the neighborhood of $\alpha$ in rectangles around the points where $\alpha$ and $\beta$ intersect, and then we construct the associated function $f_\beta$ by the same procedure as for $f_\alpha$.  We would like to define a one-parameter family of functions $S \to S^1$ by the formula $(1-t)f_\alpha + tf_\beta$ for $ 0\le t \le 1$.  This does not quite make sense as it stands since scalar multiplication is not defined for maps to $S^1$, but we can give it meaning by considering the covering space $\widetilde S$ of $S$ corresponding to the kernel of the map $\pi_1(S)\to \pi_1(S^1)$ induced by $f_\alpha$ and $f_\beta$, which are homotopic since $\alpha$ and $\beta$ both represent $x$.  Then $f_\alpha$ and $f_\beta$ lift to maps $\widetilde f_\alpha$ and $\widetilde f_\beta$ from $\widetilde S$ to $\R$, and we can form the linear combination $(1-t)\widetilde f_\alpha + t\widetilde f_\beta$. This is equivariant with respect to the action of $\Z$ as deck transformations in $\widetilde S$ and $\R$, so it passes to a well-defined map $f_t=(1-t)f_\alpha + tf_\beta$ from $S$ to $S^1$.  

The critical points of $f_t$ are the closures of the components of the complement of the union of the annular neighborhoods of $\alpha$ and $\beta$.  In the interiors of the rectangles where these neighborhoods intersect there are no critical points since the gradient vectors of $f_t$ are the vectors $(1-t)\nabla f_\alpha + t\nabla f_\beta$ which are nonzero.  Since there are finitely many complementary components of $\alpha \cup \beta$, the function $f_t$ has finitely many critical values, each of which varies linearly with $t$. For fixed $t$ the complement of the critical values consists of finitely many open intervals in $S^1$, and the preimages of these intervals consist of finitely many open annuli in $S$, thickenings of disjoint curves which are oriented transversely by $\nabla f_t$, with weights given by the lengths of the corresponding intervals in $S^1$.  These curves determine a weighted oriented multicurve representing a point $\alpha_t$ in $\A_x(S)$ by discarding any trivial curves and replacing isotopic curves by a single curve, weighted by the appropriate signed sum of the weights of the isotopic curves. For $t=0$ we have $\alpha_0 = \alpha$ and for $t=1$ we have $\alpha_1=\beta$.  The point $\alpha_t\in \A_x(S)$ varies continuously with $t$ since the functions $f_t$ vary smoothly and the intervals of noncritical values vary continuously, shrinking to length zero when critical values coalesce.  This happens only finitely often for the path $\alpha_t$ since the finitely many critical values are varying linearly with $t$.  By similar reasoning the path $\alpha_t$ varies continuously with the weights on the original multicurve $\alpha$.  Thus by fixing $\beta$ and letting $\alpha$ vary over all of $\A_x(S)$ we obtain a contraction of $\A_x(S)$.

Now we show that $\A_x(S)$ deformation retracts onto $\B_x(S)$, which implies that $\B_x(S)$ is also contractible.  The procedure here will be the same ``draining" process as in \cite{bbm}.  If a point $\sum_i k_i c_i$ in $\A_x(S)$ is not reduced, let $\{R_j\}$ be the collection of oriented compact subsurfaces of $S$ whose oriented boundary is a subset of the $c_i$'s, respecting their given orientations.  We deform $\sum_i k_i c_i$ by subtracting $t\sum_j\bdy R_j$ for increasing $t\ge 0$ until one or more $c_i$'s becomes $0$.  Deleting these $c_i$'s and the $R_j$'s whose boundaries include these $c_i$'s, we then iterate the process until we obtain a reduced weighted multicurve in $\B_x(S)$.  It is clear this process depends continuously on the initial point $\sum_i k_i c_i$ and so defines a deformation retraction of $\A_x(S)$ into $\B_x(S)$.
\end{proof}

\noindent
{\bf Canonical triangulations of \boldmath$\A_x(S)$ and \boldmath$\B_x(S)$.}  Although we will not need this in the rest of the paper, there is a canonical subdivision of the cell complex $\B_x(S)$ as a simplicial complex whose vertices are the integer points of $\B_x(S)$, the linear combinations $\sum_i k_i c_i$ with positive integer coefficients $k_i$.  This subdivision can be obtained as follows.  In the preceding proof we associated to $\sum_i k_i c_i$ a map $f\cln S \to S^1$.  This factors through the oriented metric graph $G(c)$ associated to $\sum_i k_i c_i$, with an induced map $\varphi\cln G(c)\to S^1$.  
Recall that the dimension of the cell $E(c)$ is $1$ less than the number of complementary regions of $c$, which is the number of vertices of $G(c)$.  This cell is subdivided by the various hyperplanes where two vertices of $G(c)$ have the same image under $\varphi$, hyperplanes defined by linear equations in the variables $k_i$ with integer coefficients and integer constant terms.  These hyperplanes subdivide $E(c)$ into simplices whose barycentric coordinates are the lengths of the segments of $S^1=\R/\Z$ between adjacent images of vertices of $G(c)$.  The vertices of the subdivision of $E(c)$ are thus the points $\sum_i k_i c_i$ where all vertices of $G(c)$ have the same image under $\varphi$, which is equivalent to saying that all the coefficients $k_i$ are integers.

The same procedure works more generally for $\A_x(S)$, where the noncompact cells are subdivided into infinitely many simplices.


\section{The complex of homologous curves}
\label{sec:connect}

In this section we prove Theorem~\ref{connected}, that $\C_x(S)$ is connected when $g \geq 3$ and $x$ is any nonzero primitive class in $H_1(S)$.  Recall the basic fact (see, e.g., \cite[Proposition~6.2]{fm}) that primitive classes $x$ are exactly those represented by oriented simple closed curves with coefficient $1$.  Thus when $x$ is primitive, $\C_x(S)$ is the subcomplex of $\B_x(S)$ consisting of the cells that are simplices corresponding to cycles $\sum_i k_i c_i$ where the $c_i$'s are disjoint oriented curves each representing the homology class $x$ and $\sum_i k_i =1$.  To prove that $\C_x(S)$ is connected it will suffice to show that the map $\pi_0\C_x(S) \to \pi_0\B_x(S)$ is injective when $g\ge 3$ since we already know that $\B_x(S)$ is connected.
We will show in fact that each edge path $\gamma$ in $\B_x(S)$ with endpoints in $\C_x(S)$ is homotopic, fixing endpoints, to an edge path in $\C_x(S)$.  

Deforming the edge path $\gamma$ into $\C_x(S)$ will be done by a sequence of local deformations, gradually decreasing the maximum value along $\gamma$ of the ``weight function" $W\cln \B_x(S) \to (0,\infty)$ defined by  
$$
\textstyle{W\bigl(\sum_i k_i c_i\bigr) = \sum_i k_i}
$$
In terms of dual graphs, $W$ measures the total length of all the edges.  The function $W$ is linear on cells of $\B_x(S)$ and takes integer values on vertices.  It follows that the image of $W$ is contained in $[1,\infty)$.  As $x$ is assumed to be primitive, we have $\C_x(S) = W^{-1}(1)$ since $W$ takes the value $1$ on vertices of $\C_x(S)$, hence on simplices of $\C_x(S)$, and if $W=1$ on all vertices of a cell of $\B_x(S)$ then these vertices lie in $\C_x(S)$ and span a simplex of $\C_x(S)$.  

Thus it will suffice to deform $\gamma$ to decrease the maximum value of $W$ along its vertices to $1$.  There will be two main steps.  First, when the maximum occurs at two successive vertices of $\gamma$ we will deform this edge of $\gamma$ across a $2$-cell in $\B_x(S)$  having smaller values of $W$ on all the other vertices of the cell.  Then by a more complicated procedure we will deform $\gamma$ on the two edges surrounding a vertex where $W$ is maximal to decrease the maximum along this part of $\gamma$.

\begin{proof}[Proof of Theorem~\ref{connected}]

Let $\langle v_0,v_1\rangle$ be an edge of $\B_x(S)$ joining vertices $v_0$ and $v_1$.  Associated to the points $v_t$ along this edge, $0\le t \le 1$, are dual graphs $G_t$.  These are oriented metric graphs having two vertices $a$ and $b$ for $0<t<1$, so they have the form shown in the figure below.  Since $G_t$ is recurrent, there is at least one edge from $a$ to $b$ and at least one edge from $b$ to $a$, but the number of edge-loops at $a$ or $b$ can be zero.  

\vspace{-8pt}
\begin{center}
\includegraphics[scale=1]{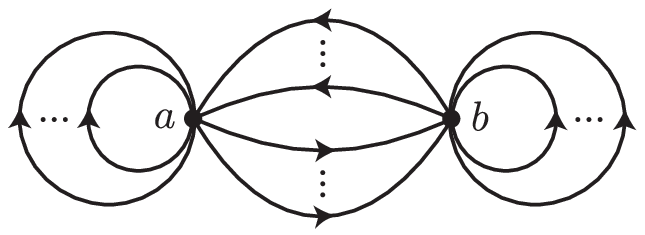}
\end{center}
\vspace{-6pt}

\noindent
As $t$ varies from $0$ to $1$, the lengths of the edges joining $a$ and $b$ vary, with all edges from $a$ to $b$ increasing in length at the same rate that all edges from $b$ to $a$ shrink, or vice versa. This corresponds to varying the weights in the weighted oriented multicurve $v_t$ by adding a multiple of the boundary of the subsurface of $S$ corresponding to $a$ or $b$.  If there are edge-loops at $a$ or $b$, their lengths do not change.  When $t$ reaches $0$ or $1$, at least one edge joining $a$ and $b$ shrinks to length $0$ and the vertices $a$ and $b$ coalesce.

Suppose that $G_t$ has at least two edges entering at vertex $a$.  We can pinch equal-length segments of two of these edges together at $a$ to produce a new metric graph $G'_t$ with three vertices.  Assuming that the subsurface of $S$ corresponding to $a$ is not simply a pair of pants, we can realize $G'_t$ as the dual graph for a point in a $2$-cell of $\B_x(S)$ where the subsurface of $S$ corresponding to the new vertex is a pair of pants, two of whose boundary curves correspond to the two edges of $G_t$ being pinched together, with the third boundary curve corresponding to the new edge of $G'_t$.  Note that pinching $G_t$ to $G'_t$ preserves the recurrence property so we do indeed have a $2$-cell of $\B_x(S)$.  In similar fashion we could realize the graph $G'_t$ obtained by pinching segments of two edges of $G_t$ exiting $a$, or two edges entering or exiting $b$. 

Let us apply this construction when $\langle v_0,v_1\rangle$ is an edge of $\gamma$ with $W(v_0)=W(v_1)$ and this is the maximum value of $W$ along $\gamma$.  The condition $W(v_0)=W(v_1)$ means that at both $a$ and $b$ there will be the same number of entering edges as exiting edges since varying $t$ does not change the total length of all the edges of $G_t$.  If the number of entering and exiting edges is equal to $1$ at both $a$ and $b$ then $G_t$ is just a circle and $W(v_0)=W(v_1)=1$ since the class $x$ is primitive.  We may thus assume the number of entering and exiting edges at $a$, say, is greater than $1$.  We can then perform the pinching operation at $a$ and this decreases the total length of $G_t$, so we obtain a $2$-cell with the function $W$ taking on its maximum value only along the edge $\langle v_0,v_1\rangle$ of this $2$-cell, since the $2$-cell is a convex polygon and $W$ is a nonconstant linear function on this polygon.  If we modify $\gamma$ by pushing the edge $\langle v_0,v_1\rangle$ across this $2$-cell to the complementary edges in the boundary of the cell, we have then improved the situation so that $W$ has strictly smaller values between $v_0$ and $v_1$.  After repeating this step finitely many times we can arrange that the maximum value of $W$ along $\gamma$ occurs only at isolated vertices. 

Now let $\langle v_0,v_1\rangle$ be an edge of $\gamma$ with $W(v_0)$ an isolated maximum of $W$ along $\gamma$.  A special case is when one of the vertices $a$ or $b$ of the graph $G_t$ for points in the interior of this edge has valence $3$ and the corresponding subsurface of $S$ is a pair of pants.  In this case we call $\langle v_0,v_1\rangle$ a \emph{P-edge}.  We wish to reduce to the case that all edges $\langle v_0,v_1\rangle$ of $\gamma$ adjacent to vertices $v_0$ with maximal $W$ value are P-edges, so suppose on the contrary that $\langle v_0,v_1\rangle$ is not a P-edge.  Since $W(v_1)<W(v_0)$, the number of edges in $G_t$ from $a$ to $b$ will be different from the number of edges from $b$ to $a$.  For whichever type of edge there are more of, we can pinch two of the edges of this type together at either $a$ or $b$, and this gives rise to a deformation of $\langle v_0,v_1\rangle$ across a $2$-cell of $\B_x(S)$ as in the preceding paragraph since we assume $\langle v_0,v_1\rangle$ is not a P-edge.  Then for the new path $\gamma$ the new edge at $v_0$ is a P-edge.  Iterating this step, we can arrange that all edges of $\gamma$ at vertices with $W$ maximum are P-edges.

Let $v_0$ be a vertex of $\gamma$ with $W$ maximal and greater than $1$.  If we cut $S$ along the multicurve given by $v_0$ we obtain a cobordism $R$ between two copies of this multicurve, which we label $\bdy_+ R$ and $\bdy_- R$.  A P-edge from $v_0$ then corresponds to a pair of pants in $R$ with two boundary curves in $\bdy_+ R$ or two boundary curves in $\bdy_- R$. If the genus of $S$ is at least $3$, such a pair of pants is uniquely determined by its third boundary curve, which gives a vertex in the curve complex $\C(R)$. Such vertices span a subcomplex $\PP(R)$ of $\C(R)$.  We will prove below that $\PP(R)$ is connected when the genus of $S$ is at least $3$.  Assuming this, we finish the proof of the theorem as follows.  By the preceding paragraph, we can assume the edges in $\gamma$ on both sides of $v_0$ are P-edges.  Since $\PP(R)$ is connected, we can interpolate between these two P-edges a sequence of P-edges, each corresponding to a pair of pants disjoint from the next one. Each pair of successive P-edges then forms two adjacent edges of a $2$-cell of $\B_x(S)$ (either a triangle or a square).  We then deform $\gamma$ by pushing across each of these $2$-cells, thereby replacing the two edges of $\gamma$ adjacent to $v_0$ by a sequence of edges along which $W$ has values smaller than $W(v_0)$ since $W$ is linear on each of the $2$-cells. 

After finitely many iterations of these steps we eventually deform $\gamma$, staying fixed on its endpoints, to a path in $W^{-1}(1)=\C_x(S)$.
\end{proof}

\begin{lemma}
\label{pp}
The complex $\PP(R)$ is connected when $g\ge 3$.
\end{lemma}

\begin{proof}
Instead of regarding vertices of $\PP(R)$ as isotopy classes of pairs of pants in $R$ we can regard them as isotopy classes of arcs in $R$ joining two curves of $\bdy_-R$ or two curves of $\bdy_+R$, where the pair of pants corresponding to such an arc is a thickening of the union of the arc and the two curves at its endpoints.  When we say ``arc" in what follows, we will mean an arc giving a vertex of $\PP(R)$ in this way.  A simplex of $\PP(R)$ corresponds to a collection of disjoint arcs joining disjoint pairs of curves of $\bdy R$.  Let us fix a standard arc $a_0$ joining two curves in $\bdy_-R$.  An arbitrary arc $a$ can be connected to an arc joining the same two curves of $\bdy_-R$ as $a_0$ by a sequence of at most two edges of $\PP(R)$, first by choosing an arc $a'$ disjoint from $a$ joining curves of $\bdy_+R$ if $a$ does not already do this, then by choosing an arc disjoint from $a'$ joining the two curves of $\bdy a_0$.  Thus it suffices to connect an arbitrary arc $a$ joining the two curves of $\bdy_-R$ containing $\bdy a_0$ to the arc $a_0$ by a sequence of edges in $\PP(R)$.

First we do this for three special types of arcs.  For the first two cases we fix a genus $0$ subsurface $R_+$ of $R - a_0$ containing $\bdy_+R$ and having just one more boundary curve, which lies in the interior of $R$.

\vfill\eject

\begin{center}
\includegraphics[scale=1]{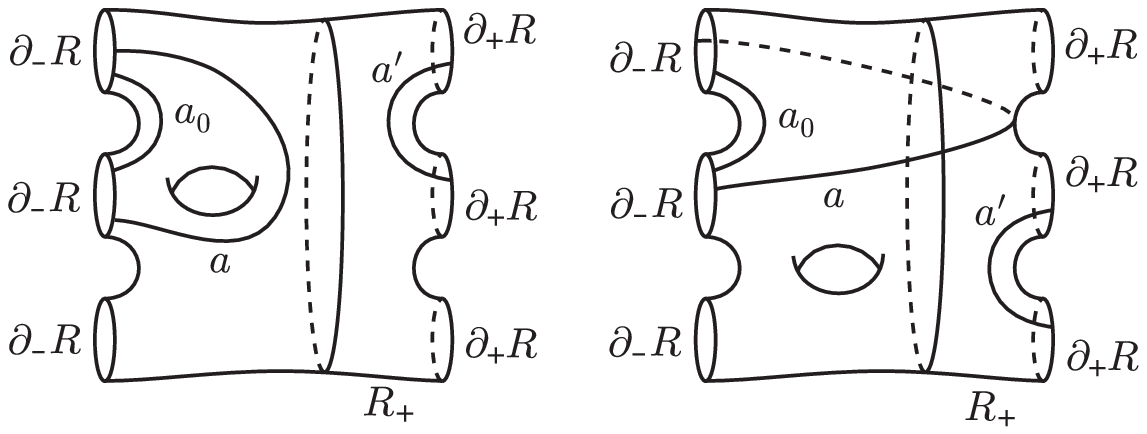}
\end{center}

\smallskip\noindent
({\it i\/}) An arc $a$ that is disjoint from $R_+$ as in the first figure above.  In this case there is an edge joining $a$ to an arc $a'$ in $R_+$ and then an edge joining $a'$ to $a_0$.

\smallskip\noindent
({\it ii\/}) An arc $a$ that intersects $R_+$ in a single subsegment that separates $R_+$ into two components, one of which contains only one curve of $\bdy_+R$ and the other of which contains at least two curves of $\bdy_+R$, assuming that $\bdy_+R$ has at least three curves in total.  This case is illustrated in the second figure above.  In this case we can choose an arc $a'$ in $R_+$ disjoint from $a$ and proceed as in ({\it i\/}).

\smallskip\noindent
({\it iii\/}) An arc $a$ as in the figure below, in the case that $\bdy_+R$ contains just two curves, hence $R$ has genus at least $1$ since $g\ge 3$. Then there are edges joining $a$ to $a'$ and then to $a_0$.

\begin{center}
\includegraphics[scale=1]{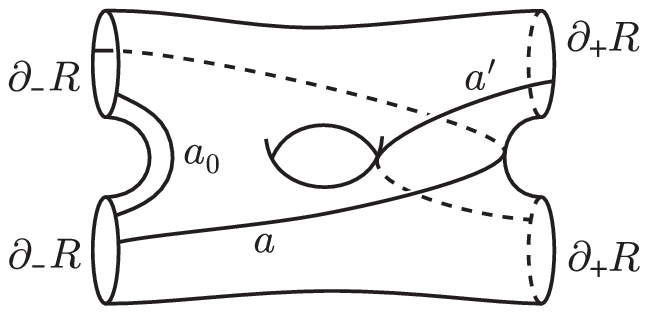}
\end{center}

Now we reduce the general case to these special cases.  Let us regard the curve of $\bdy_-R$ at one end of $a_0$ as a puncture $p$ rather than a boundary component.  If we allow this puncture to move around anywhere in the surface $R'$ obtained from $R$ by filling in this puncture (or equivalently, collapsing the boundary component of $R$ to a point), then any other arc $a$ with the same endpoints as $a_0$ can be isotoped to $a_0$.  This implies that $a_0$ can be transformed to $a$ (or an arc isotopic to $a$) by a diffeomorphism $h$ of $R$ obtained by dragging $p$ around a loop in $R$ based at $p$.  Such diffeomorphisms form a subgroup of $\Mod(R',p)$, the image of the boundary map $d\cln \pi_1(R',p)\to \pi_0\Diff^+(R',p)$ in the long exact sequence of homotopy groups associated to the fibration $\Diff^+(R',p)\to\Diff^+(R')\to R'$ obtained by evaluating diffeomorphisms of $R'$ at $p$, with fiber the subgroup $\Diff^+(R',p)$ of $\Diff^+(R')$ consisting of diffeomorphisms that fix $p$.  The map $d$ associates to a loop at $p$ the diffeomorphism obtained by dragging $p$ around this loop.  (The long exact sequence in fact reduces to a short exact sequence, the Birman exact sequence.)

Since $d$ is a homomorphism, it follows that the diffeomorphism $h$ is isotopic to a composition $h_1 \cdots h_n$ of diffeomorphisms $h_i$ obtained by dragging $p$ around a suitable sequence of loops that generate $\pi_1(R',p)$.  For such generators we can choose loops in $R-R_+$, producing arcs $a=h_i(a_0)$ as in ({\it i\/}), or loops that wind once around a single curve of $\bdy_+R$, producing arcs $a=h_i(a_0)$ as in ({\it ii\/}) or ({\it iii\/}).  By the special cases ({\it i--iii\/}) there is an edge path in $\PP(R)$ joining $a_0$ to $h_i(a_0)$, for each $i$.  By applying the product $h_1\cdots h_{i-1}$ to this edge path we obtain an edge path joining $h_1\cdots h_{i-1}(a_0)$ to $h_1\cdots h_i(a_0)$.  Stringing these edge paths together, we obtain an edge path from $a_0$ to the arc $h(a_0)=a$. 
\end{proof}


\section{The inductive step}
\label{sec:proof}

We will prove Theorem~\ref{main} by induction on genus.  In this section we give the inductive step, deferring the base case of genus $2$ until the next section since it requires methods that are special to that case.

The inductive step will use the following basic fact about group actions:
\begin{quote}
\emph{Suppose a group $G$ acts on a connected cell complex $X$.  Let $A \subseteq G$ be a subset with the property that, for any two vertices of $X$ connected by an edge, there is an element of $A$ taking one vertex to the other.  Then $G$ is generated by the union of $A$ and the set of vertex stabilizers.}
\end{quote}

\noindent
Since $\C_x(S)$ is connected when $g\ge 3$, we may apply this fact to the case of the $\I(S)$ action on $\C_x(S)$.  For $A$ we choose the set of bounding pair maps and twists about separating curves in $\I(S)$.  
The condition on edges is verified in the next lemma.  Here and in what follows we use the notation $T_a$ for the twist along the curve $a$.

\begin{lemma}
\label{bp}
If $v$ and $w$ are vertices of $\C_x(S)$ that are connected by an edge, then there is a bounding pair map $T_a^{}T_b^{-1}$ in $\I(S)$ with $T_a^{}T_b^{-1}(v)=w$.
\end{lemma}

It is worth pointing out that this lemma together with the connectedness of $\C_x(S)$ immediately implies the non-obvious fact, known to Johnson \cite[page 253, line 6]{djconj}, that when $g\ge 3$, any two oriented curves in $S$ in the same homology class are equivalent under the action of the Torelli group.  This also holds when $g=2$, as we will see in the next section.

\begin{proof}

We can view $S$ as a torus with handles attached, with $v$ and $w$ as longitudes on the torus.  Then if we choose $a$ and $b$ as meridians on the torus as shown in the figure below, it is clear that $T_a^{}T_b^{-1}$ takes $v$ to $w$.  
\end{proof}

\begin{center}
\includegraphics[scale=1]{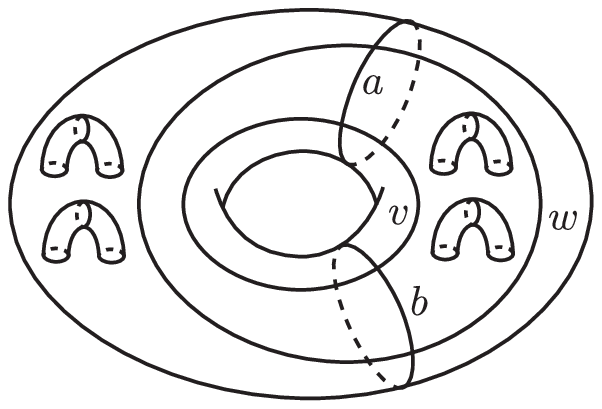}
\end{center}

To finish the inductive step it remains to show that for each vertex $v$ of $\C_x(S)$, the stabilizer of $v$ in $\I(S)$ is contained in the subgroup of $\I(S)$ generated by bounding pair maps and twists about separating curves.  If $v$ is represented by an oriented curve $a$ in the homology class $x$, the stabilizer of $v$ in $\I(S)$ is the subgroup $\I(S,a)$ represented by diffeomorphisms leaving $a$ invariant. 
More generally, let $\Mod(S,a)$ denote the stabilizer of $a$ in $\Mod(S)$.  There is a natural homomorphism $\eta \cln \Mod(S,a) \to \Mod(S',P)$ where $S'$ is the closed surface of genus $g-1$ obtained from $S$ by cutting along $a$ and collapsing the resulting two boundary curves to a pair $P=\{p,q\}$ of distinguished points, and $\Mod(S',P)$ is the mapping class group of $S'$ fixing each of these two points.  The kernel of $\eta$ is the infinite cyclic group generated by $T_a$ \cite[Proposition 3.20]{fm}.  Since $H_1(S',P)=H_1(S,a)=H_1(S)/\langle [a]\rangle$, and since $T_a^k \notin \I(S)$ for $k \neq 0$, it follows that $\eta$ restricts to an isomorphism:
\[ \eta : \I(S,a) \to \I(S',P).\]
where the latter group is the kernel of the natural homomorphism $\Mod(S',P)\to \Aut(H_1(S',P))$. 

If $T_d$ is a twist about a separating curve $d$ in $S'-P$, then $\eta^{-1}(T_d)$ is either a bounding pair map $T_d^{}T_a^{-1}$ or a twist about a separating curve $T_d$, depending on whether or not $d$ separates the two points of $P$ in $S'$.  If $T_c^{}T_d^{-1}$ is a bounding pair map in $\I(S',P)$, then $c \cup d$ does not separate the two points of $P$, otherwise $T_c^{}T_d^{-1}$ would not preserve the homology class of an arc that intersects $c$ once and is disjoint from $d$.  It follows that $\eta^{-1}(T_c^{}T_d^{-1})$ is a bounding pair map in $\I(S,a)$.  Thus in order to show that $\I(S,a)$ is generated by twists about separating curves and bounding pair maps it suffices to show that $\I(S',P)$ has this property.

We will prove this by considering two short exact sequences:
\[ 1 \to K_1 \to \I(S',p) \to \I(S') \to 1 \]
\[ 1 \to K_2 \to \I(S',P) \to \I(S',p) \to 1 \]
These are analogs for the Torelli group of Birman exact sequences for the full mapping class group; the first was considered by Johnson \cite[Lemma~3]{dj1} and the second by van den Berg \cite[Proposition~2.4.1]{vdB} and Putman \cite[Theorem~4.1]{ap}.

Consider first the first sequence.  Here $\I(S',p)$ is the subgroup of $\Mod(S',p)$ acting trivially on $H_1(S',p) = H_1(S')$.  The map $\I(S',p)\to\I(S')$ forgets the point $p$.  Elements of $\I(S')$ represented by separating twists and bounding pair maps clearly lift to such elements of $\I(S',p)$.  We claim the kernel $K_1$ is generated by bounding pair maps.  As in the proof of Lemma~\ref{pp}, the kernel of $\Mod(S',p)\to\Mod(S')$ consists of elements obtained by dragging $p$ around a loop in $S'$.  This is the image of the homomorphism $d\cln \pi_1(S',p) \to \pi_0\Diff^+(S',p)$ defined there.  Generators for this kernel are obtained by dragging $p$ around generators for $\pi_1(S',p)$, and we can choose embedded nonseparating curves for these generators, from the standard presentation of $\pi_1(S',p)$.  Such drag maps are bounding pair maps, so the kernel of $\Mod(S',p)\to\Mod(S')$ is in fact contained in $\I(S',p)$ and so coincides with $K_1$.  Thus $K_1$ is generated by bounding pair maps, as claimed.  It follows that $\I(S',p)$ is generated by separating twists and bounding pair maps if this is true for $\I(S')$.

Now we proceed to the second short exact sequence.  Separating twists and bounding pair maps in $\I(S',p)$ lift to such maps in $\I(S',P)$ by choosing the point $q$ to be sufficiently close to $p$.  Next we show that the kernel $K_2$ is generated by products of separating twists in $\I(S',P)$.  The kernel of $\Mod(S',P)\to\Mod(S,p)$ is formed by maps $d_\gamma = d(\gamma)$ resulting from dragging $q$ around loops $\gamma$ in $S'-p$, where now $d$ is the boundary map $\pi_1(S'-p,q)\to \pi_0\Diff^+(S',P)$.  To see how $d_\gamma$ acts on $H_1(S',P)$, consider the short exact sequence 
\[ 0\to H_1(S')\to H_1(S',P)\to \widetilde{H}_0(P)\to 0   \]
which is invariant under $\Mod(S',P)$.  Clearly $d_\gamma$ acts trivially on the image of $H_1(S')$ in $H_1(S',P)$ since $d_\gamma$ is trivial in $\Mod(S')$.  Also $d_\gamma$ acts trivially on $\widetilde{H}_0(P)$ since it fixes $P$.  The action of $d_\gamma$ on $H_1(S',P)$ will then be trivial if and only if it acts trivially on an element of $H_1(S',P)$ that maps to a generator of $\widetilde{H}_0(P)$. We can represent such an element by an arc $\epsilon$ joining $p$ to $q$.  It follows from the fact that the diffeomorphisms $d_\gamma$ act trivially on $H_1(S')$ and $\widetilde{H}_0(P)$ that there is a homomorphism $\varphi\cln\pi_1(S'-p,q)\to H_1(S')\subset H_1(S',P)$ such that $d_\gamma[\epsilon] = [\epsilon] + \varphi(\gamma)$ for all $\gamma$.  We claim that $\varphi(\gamma)=[\gamma]$, so that $\varphi$ is the abelianization map $\pi_1(S'-p,q)\to H_1(S'-p)=H_1(S')$.  Since $\varphi$ is a homomorphism it suffices to check that $d_\gamma[\epsilon]=[\epsilon]+[\gamma]$ for $\gamma$ ranging over a set of generators for $\pi_1(S'-p,q)$.  As generators we can choose embedded loops $\gamma$ disjoint from $\epsilon$, and for such loops $\gamma$ the formula $d_\gamma[\epsilon]=[\epsilon]+[\gamma]$ obviously holds.

From the formula $d_\gamma[\epsilon]=[\epsilon]+[\gamma]$ we see that $d_\gamma$ acts trivially on $H_1(S',P)$ if and only if $\gamma$ lies in the commutator subgroup of $\pi_1(S'-p)$.  It is a general fact that the commutator subgroup of a group is normally generated by commutators of generators of the group.  For $\pi_1(S'-p)$ we choose generators coming from representing $S'-p$ as a punctured $4(g-1)$-gon with opposite edges identified (this is not the standard identification!). These generators are nonseparating curves, any two of which intersect transversely in one point, so their commutator is represented by a curve $\gamma$ bounding a genus $1$ subsurface of $S'-p$.  The map $d_\gamma$ is then the composition of a twist along a parallel copy of $\gamma$ and an inverse twist along another parallel copy of $\gamma$.  This shows that $K_2$ is generated by products of separating twists.  All such separating twists lie in $\I(S',P)$ since for a twist along a separating curve $\gamma$ that does not separate $p$ and $q$ a basis for $H_1(S,P)$ can be chosen disjoint from $\gamma$, while if $\gamma$ does separate $p$ and $q$ then composing the twist along $\gamma$ with $d_\gamma$ or $d_\gamma^{-1}$ converts the twist along $\gamma$ to a twist along a $\gamma$ that does not separate $p$ and $q$, and $d_\gamma\in\I(S',P)$.

In summary, we have shown the inductive step:

\begin{prop}
If $\I(S')$ is generated by separating twists and bounding pair maps, then so are $\I(S',p)$, $\I(S',P)$, and $\I(S)$ in turn, in the last case assuming that the genus of $S$ is at least $3$.
\end{prop}

The next two propositions, due to Johnson \cite{djhomeo}, justify two supplementary statements made earlier. 

\begin{prop}
\label{seps}
Every Dehn twist about a separating curve in a surface of genus at least $3$ is isotopic to a product of bounding pair maps.
\end{prop}

\begin{proof}

First consider the case that $S$ has genus $3$. A nontrivial separating curve $c$ in $S$ then splits $S$ into a punctured torus and a punctured genus $2$ surface, and the latter surface can be further decomposed as the union of a $4$-punctured sphere and a pair of pants, as in the figure below.

\vfill\eject

\begin{center}
\includegraphics[scale=1]{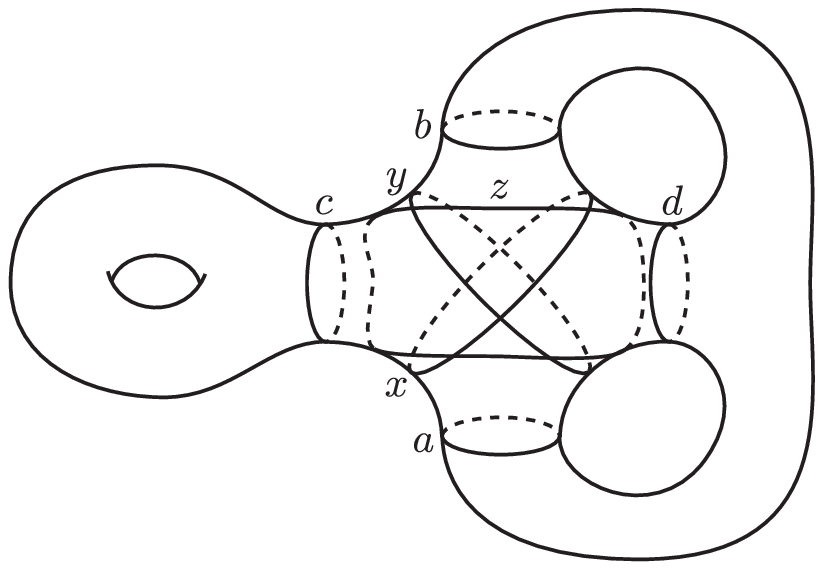}
\end{center}

The lantern relation \cite[p. 187]{dehn} \cite[\S IV]{djhomeo} \cite[\S 5.1]{fm}
gives
\[ T_xT_yT_z = T_aT_bT_cT_d \]
Since each of $T_a$, $T_b$, $T_c$, and $T_d$ commutes with all seven Dehn twists in the relation, we can rewrite this relation as follows:
\[ (T_x^{}T_a^{-1})(T_y^{}T_b^{-1})(T_z^{}T_d^{-1}) = T_c^{}. \]
In other words the Dehn twist about the separating curve $c$ is the product of three bounding pair maps.  This takes care of the genus $3$ case.  The general case is obtained from this by attaching the appropriate number of handles to the punctured torus and the pair of pants.
\end{proof}

\begin{prop}
\label{bps}
Every bounding pair map of a surface $S$ is isotopic to a product of bounding pair maps associated to pairs of curves bounding genus $1$ subsurfaces of $S$.
\end{prop}

\begin{proof}

A bounding pair map is obtained by twisting a subsurface bounded by two curves through $360$ degrees.  If this subsurface has genus $n$, it can be decomposed into $n$ subsurfaces of genus $1$, each bounded by two curves, as indicated in the figure below, and the twist of the genus $n$ subsurface is isotopic to the composition of $n$ twists of the genus $1$ subsurfaces. 
\end{proof}

\vspace{4pt}
\begin{center}
\includegraphics[scale=1]{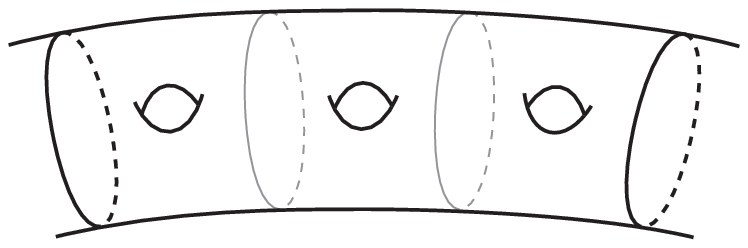}
\end{center}


\vfill\eject

\section{Starting the induction: genus 2}
\label{sec:2}

Now we restrict to the case that $S$ has genus $2$, where we want to show that $\I(S)$ is generated by separating twists.  The complex $\C_x(S)$ is $0$-dimensional in this case, so is of little help.  Instead we use $\B_x(S)$, which is $1$-dimensional and contractible, hence a tree.  From the elementary theory of groups acting on trees, it will suffice to show that the quotient $\B_x(S)/\I(S)$ is also a tree, and that the stabilizers of vertices in $\B_x(S)$ are generated by separating twists.  

It will be helpful to know exactly what an edge of $\B_x(S)$ looks like.  Such an edge corresponds to a multicurve separating $S$ into two components, so these components must be pairs of pants, with the multicurve consisting of three nonseparating curves $a,b,c$ as in the figure below.  

\begin{center}
\includegraphics[scale=1]{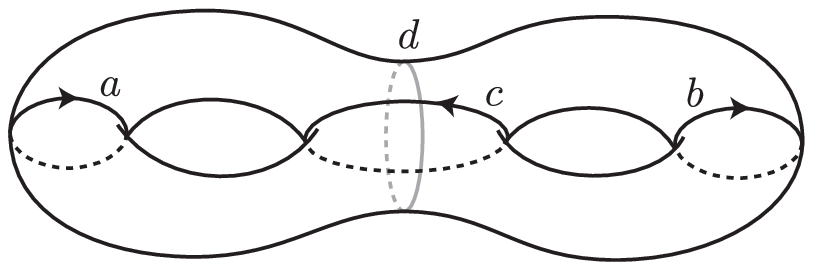}
\end{center}

\noindent
Since the multicurve is reduced, two of the three curves, say $a$ and $b$, will be oriented consistently with an orientation of either of the pairs of pants, and the third curve $c$ will be oppositely oriented.  As we move across the edge we transfer weights from $a$ and $b$ to $c$, or vice versa.  Transferring weights from $a$ and $b$ to $c$ decreases the value of $W$.  At the end of the edge with larger $W$-value the weighted multicurve is $pa+qb$, with $p \ge q$ say, and then at the other end the weighted multicurve is $(p-q)a + qc$, so we subtract $q$ from the weights on $a$ and $b$ and add $q$ to the weight on $c$.  The value of $W$ decreases from $p+q$ to $p$.  (Thus for a sequence of edges along which $W$ decreases, the pairs of weights are changing according to the Euclidean algorithm of repeatedly subtracting the smaller of two numbers from the larger.)

We claim that from a given vertex $pa+qb$ with $W>1$ (so both $p$ and $q$ are greater than $0$) all the edges of $\B_x(S)$ leading to vertices with smaller $W$-value are equivalent under the action of the stabilizer of the vertex in $\I(S)$.  This implies that each component of $\B_x(S)/\I(S)$ is a tree since there is a well-defined flow on it decreasing the values of $W$ monotonically until they reach the value $1$ at a vertex represented by a single curve. Since $\B_x(S)$ is connected, so is its quotient $\B_x(S)/\I(S)$, so the quotient must then be a tree, with a single vertex where $W=1$.  In particular, this shows that $\I(S)$ acts transitively on oriented curves in a given homology class, just as in higher genus as we noted in the remarks following Lemma~\ref{bp}.

To verify the claim, consider two edges of $\B_x(S)$ leading downward from this vertex.  These two edges correspond to two different choices for the $c$ curve, in the notation above.  Both choices lie in the complement of the $a$ and $b$ curves, a $4$-punctured sphere.  Isotopy classes of nontrivial curves in a $4$-punctured sphere are classified by their slope, an element of $\Q\cup\{\infty\}$.  Let us choose coordinates on the $4$-punctured sphere so that one choice of $c$ has slope $0/1$ and a separating curve $d$ as in the preceding figure has slope $1/0$.  The mapping class group of a $4$-punctured sphere fixing each of the punctures can be identified with the subgroup $G$ of $PSL(2,\Z)$ represented by matrices congruent to the identity mod~$2$.  The action of $G$ on slopes has three orbits, the slopes whose numerators and denominators are congruent to those of $0/1$, $1/0$, or $1/1$ mod $2$.  Topologically, these three classes are distinguished by how the corresponding curves separate the $4$ punctures into pairs.  In particular, the slopes congruent to $1/0$ correspond to separating curves on $S$, such as the curve $d$.  Whether the value of $W$ decreases or increases for a particular choice for $c$ depends only on how $c$ separates the $4$ punctures.  For $c$ of slope $0/1$ as in the figure the value of $W$ decreases, so the slopes congruent to $0/1$ correspond to the curves $c$ for which $W$ decreases.  These slopes form the vertices of a tree $T$ that can be visualized by superimposing it on the Farey diagram, as in the figure below.  

\begin{center}
\includegraphics[scale=1]{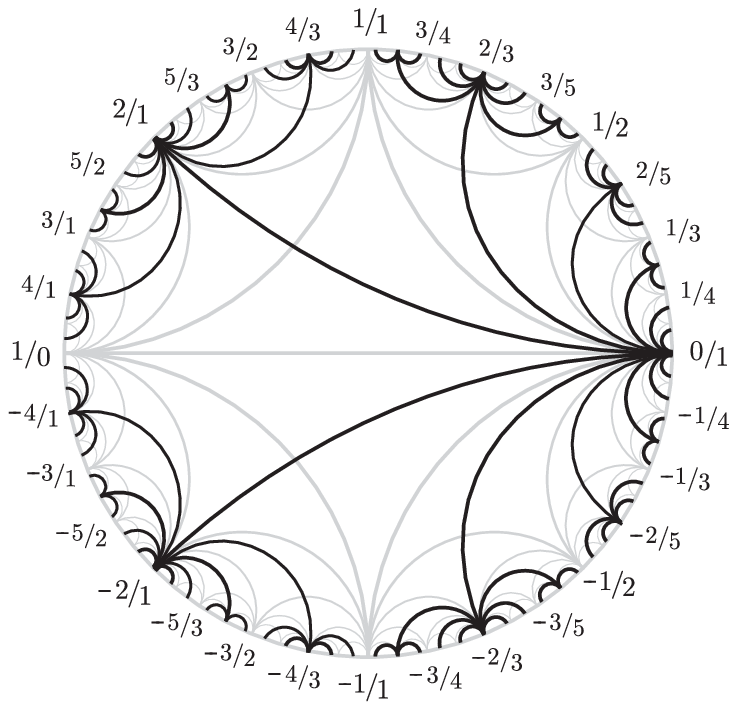}
\end{center}

Two vertices of $T$ are joined by a sequence of edges of $T$, and the curves corresponding to vertices along this path are related each to the next by Dehn twists along curves of slope congruent to $1/0$ mod $2$.  For example, a twist along a curve of slope $1/0$ takes slope $0/1$ to slope $2/1$.  Slopes congruent to $1/0$ mod $2$ give separating curves in $S$, so the path in $T$ gives a product of separating twists on $S$ taking the first choice of $c$ to the second choice, verifying the claim.

Next we check that the stabilizer of a vertex of $\B_x(S)$ corresponding to a vertex $pa+qb$ with $p,q>0$ is generated by separating twists.  We have just seen that products of separating twists in the stabilizer act transitively on vertices of $T$, so, modulo such products, we can assume the element of the stabilizer fixes the slope $0/1$ curve $c$.  But the stabilizer of the multicurve $a\cup b \cup c$ in $\I(S)$ is trivial, since an element in this stabilizer would have to be a product of twists along $a$, $b$, and $c$, and it is easy to see that such a product acts nontrivially on homology unless it is the trivial product.  (Consider the action on curves that intersect two of $a,b,c$ transversely in one point and are disjoint from the third.)  Thus the stabilizer of $pa+qb$ in $\I(S)$ is generated by separating twists when both $p$ and $q$ are positive.

There remains the stabilizer of a vertex corresponding to a single curve $a$.  This situation was analyzed in the previous section for arbitrary genus, where we showed that the stabilizer of a curve in genus $g$ is generated by twists and bounding pair maps if this is true for the full Torelli group in one lower genus.  In the present situation the Torelli group for genus $1$ is trivial, so the stabilizer of a curve in genus $2$ is generated by separating twists since there are no bounding pair maps until genus $3$.

This finishes the proof that $\I(S)$ is generated by separating twists in genus $2$, and hence also the proof of the Birman--Powell theorem.

\medskip\noindent
{\bf The structure of the Torelli group in genus 2.} It is not hard to extend the preceding arguments to see that $\I(S)$ is a nonfinitely generated free group in genus $2$.  For the action of $\I(S)$ on the tree $\B_x(S)$ the edge stabilizers are trivial as we observed above, and the quotient $\B_x(S)/\I(S)$ is a tree, so $\I(S)$ is a free product of vertex stabilizers, with one factor for each vertex of $\B_x(S)/\I(S)$.  For a vertex $pa+qb$ of $\B_x(S)$ with $p,q>0$ the stabilizer acts freely on the tree $T$ since we saw that no vertices can be fixed points, and no edge can be inverted since elements of the group $G$ cannot interchange slopes congruent to $1/0$ and $1/1$ mod $2$.  The stabilizer group acts transitively on vertices of $T$, so the stabilizer is the fundamental group of the orbit space of the action, an infinite wedge of circles since vertices of $T$ have infinite valence.  Thus the stabilizer of $pa+qb$ is a free group on an infinite number of generators. 

The other case is the stabilizer of a vertex that is a single curve.  From the discussion in the preceding section this is isomorphic to the kernel $K_2$ of the map $\I(S,P)\to\I(S,p)$.  This is a subgroup of the kernel of the map $\Mod(S,P)\to\Mod(S,p)$ which is $\pi_1(S-p)$ from the long exact sequence of homotopy groups for the fibration $\Diff^+(M,p)\to M-p$ obtained by evaluating diffeomorphisms at $q$, with fiber $\Diff^+(M,P)$.  The group $\pi_1(M-p)$ is free so the subgroup $K_2$ is free as well.  It is in fact the commutator subgroup, as our analysis showed, so it is nonfinitely generated.

Since $\B_x(S)/\I(S)$ is an infinite tree, we see that $\I(S)$ is a free product of infinitely many stabilizers, each of which is a nonfinitely generated free group, so $\I(S)$ is a nonfinitely generated free group itself.  Mess \cite{gm} gives the following more precise description of the infinite generating set.  Each nontrivial separating curve in $S$ induces a splitting of $H_1(S)$ into two symplectic subspaces, namely, the subspaces consisting of the elements represented by $1$-cycles supported entirely on one side of the separating curve or the other.  Mess proved that $\I(S)$ has a free generating set where there is one generator for each such symplectic splitting of $H_1(S)$ (this description can in fact be deduced by sharpening the argument given above, as in \cite[Section 7]{bbm}).  It is not true, however, that if we make an arbitrary choice of Dehn twist for each symplectic splitting, then we obtain a generating set.  As such, it is still an open problem to turn Mess's description of the generating set into an explicit generating set.

\medskip\noindent
{\bf Final remark.}  It is tempting to try to prove the Birman--Powell theorem in genus 2 using the same inductive step we used in higher genus.  Indeed, the genus one Torelli group is trivial, so by the Birman exact sequence it would suffice to show that $\C_x(S)$, or some variant, is connected in genus 2.

We have already mentioned that $\C_x(S)$ has no edges in genus 2.  In particular, the complex is not connected.  One might try to repair this by enlarging $C_x(S)$ to a complex with edges joining pairs of vertices corresponding to curves that are not disjoint but intersect in the minimum number of points, namely $4$, such as the curves $c$ and $T_d(c)$ where $c$ and $d$ are the curves shown in the figure at the beginning of this section.  This does not work, however.  The curves $c$ and $T_d(c)$ are joined by an edge path in $\B_x(S)$ of length $2$ with $a+b$ as the intermediate vertex, so the values of $W$ along this edge path lie between $1$ and $2$.  But there are pairs of vertices of $\C_x(S)$ for which the values of $W$ along the path in $\B_x(S)$ joining the two vertices must exceed any preassigned number $n$, since one can start with a vertex of $\B_x(S)$ where $W$ has a value larger than $n$ and then follow two different paths from this vertex along which $W$ decreases monotonically until one reaches a pair of vertices in $\C_x(S)$ with $W=1$.  Since $\B_x(S)$ is a tree, these two vertices cannot be joined by any other path along which $W$ has the maximum value $2$, so these two vertices cannot be in the same path component of the proposed enlargement of $\C_x(S)$.  This argument shows moreover that $\C_x(S)$ cannot be made connected by adding only a finite number of types of edges.


\bibliographystyle{plain}
\bibliography{igen}

\end{document}